\documentclass[english]{extarticle}
\usepackage[T1]{fontenc}
\usepackage[latin9]{inputenc}
\usepackage{xcolor}
\usepackage{babel}
\usepackage{float}
\usepackage{mathrsfs}
\usepackage{amsmath}
\usepackage{amsthm}
\usepackage{amssymb}
\usepackage{graphicx}
\usepackage{geometry}
\usepackage{microtype}
\usepackage[pdfusetitle,
 bookmarks=true,bookmarksnumbered=true,bookmarksopen=false,
 breaklinks=false,pdfborder={0 0 1},backref=false,colorlinks=true]
 {hyperref}
\hypersetup{
 linkcolor=blue,urlcolor=blue,citecolor=blue,pdfstartview={FitH},hyperfootnotes=false}

\makeatletter
\numberwithin{equation}{section}
\numberwithin{figure}{section}

\usepackage{amsmath}
\usepackage{amssymb}
\usepackage{amsthm}
\usepackage{mathtools}
\usepackage{graphicx}
\usepackage{tikz-cd}
\usepackage{mlmodern}
\usepackage[T1]{fontenc}
\usepackage{xcolor}
\usepackage{pdfrender}
\usepackage{mathtools}        
\usepackage{bm}               
\usepackage{mathrsfs}         
\usepackage{dsfont}           
\usepackage{cancel}           
\usepackage{graphicx}
\usepackage{booktabs}

\AtBeginDocument{

}

\makeatother

\theoremstyle{plain}
\newtheorem{thm}{\protect\theoremname}
\theoremstyle{definition}
\newtheorem{defn}[thm]{\protect\definitionname}
\theoremstyle{remark}
\newtheorem{rem}[thm]{\protect\remarkname}
\theoremstyle{plain}
\newtheorem{prop}[thm]{\protect\propositionname}
\newtheorem{lem}[thm]{\protect\lemmaname}
\newtheorem{cor}[thm]{\protect\corollaryname}
\providecommand{\corollaryname}{Corollary}
\providecommand{\definitionname}{Definition}
\providecommand{\lemmaname}{Lemma}
\providecommand{\propositionname}{Proposition}
\providecommand{\remarkname}{Remark}
\providecommand{\theoremname}{Theorem}

\begin{document}
\title{Brownian Convergence of Planar Domains and Stability of the Planar
Skorokhod Embedding Problem}
\author{Maher Boudabra \thanks{King Fahd University of Petroleum and Minerals, KSA.},
Mrabet Becher \thanks{Monastir Preparatory Engineering Institute, Tunisia.},
Fathi Haggui \thanks{Monastir Preparatory Engineering Institute, Tunisia.}}
\maketitle
\begin{abstract}
We present a numerical framework for approximating the $\mu$-domain
in the planar Skorokhod embedding problem PSEP, recently introduced
in \cite{gross2019}. We show that under weak convergence of a sequence
of probability measures $(\mu_{n})_{n}$, the corresponding sequence
of $\mu_{n}$-domains converges, in an appropriate sense, to the domain
associated with the limit measure $\mu$. In addition, we provide
implementation strategies, convergence rate estimates, and a numerical
example. The method is robust and versatile, offering a concrete computational
approach for the approximation of $\mu$-domains. As part of this
analysis, we introduce a novel mode of convergence for planar domains
via planar Brownian motion, which we call $p$-Brownian convergence. 
\end{abstract}

\section{Introduction and motivation}

In 2019, R.~Gross introduced a planar version of the Skorokhod embedding
problem PSEP \cite{gross2019}, originally formulated in 1961 in one
dimension. See \cite{Obloj2004} for a concise survey. The planar
question studied by Gross is as follows: given a probability measure
$\mu$ with zero mean and finite second moment, does there exist a
simply connected domain $U$ (containing the origin) such that, if
$(Z_{t})_{t\geq0}$ is a standard planar Brownian motion and $\tau$
is the exit time from $U$, then $\Re(Z_{\tau})$ has distribution
$\mu$? Gross gave an affirmative answer via a smart and explicit
construction of the domain. One year later, Boudabra and Markowsky
published two papers \cite{boudabra2019remarks,Boudabra2020} on the
problem. In the first, they showed that the problem is solvable for
any distribution with a finite $p$-th moment whenever $p>1$, hence
extending Gross' technique to cover all such distributions. They also
established a uniqueness criterion. In the second, they introduced
a new class of domains that solve the Skorokhod embedding problem
and provided another uniqueness criterion. They coined the term \emph{$\mu$-domain}
to denote any simply connected domain that solves the problem. In
that work, the statement was altered to incorporate the finiteness
of $\mathbf{E}(\tau^{\frac{p}{2}})$ as a condition. Note that, in
Gross' original work ($p=2$), the finiteness of $\mathbf{E}(\tau)$
was a consequence of his construction. In \cite{BOUDABRA2026}, the
author addressed the case $p=1$, closing that line of investigation.
Note that the two constructions of Gross and Boudabra-Markowsky lead
to different geometries. 

\begin{figure}[H]
\centering \includegraphics[width=12cm,totalheight=12cm,keepaspectratio]{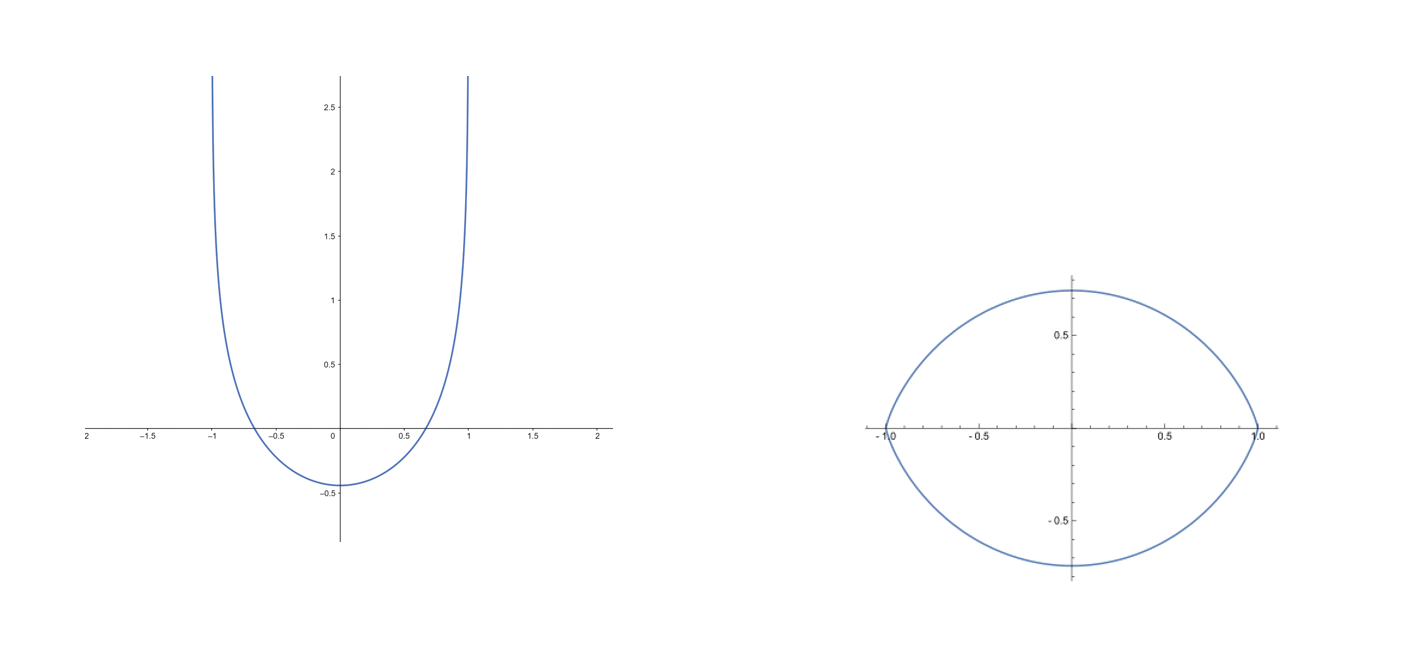}\caption{{\small For the uniform distribution on $(-1,1)$, the left domain
is Boudabra--Markowsky's solution while Gross' solution is on the
right.}}
\end{figure}

If $(a,b)$ is not charged by $\mu$, i.e., $(a,b)$ is a null set
for $\mu$ \footnote{This means that the c.d.f.\ of $\mu$ is constant on $(a,b)$.},
then any $\mu$-domain must contain the vertical strip $(a,b)\times(-\infty,+\infty)$.
This is a universal property regardless of the construction.

For the sake of convenience, we explain Gross' approach to solve the
PSEP. First, recall the basic definitions of the c.d.f. and its underlying
quantile function. The c.d.f.\ of $\mu$ is defined for all $x\in\mathbb{R}$
by
\[
F(x)=\mu((-\infty,x]).
\]
The quantile function is 
\[
q(u)=\inf\{x:\,F(x)\ge u\},\quad u\in(0,1).
\]
$q$ is then the pseudo-inverse of $F$. In particular $q=F^{-1}$
when $F$ is one to one. The quantile function plays a central role
in statistics: when fed with inputs uniformly distributed in $(0,1)$,
$q$ will generates values sampling as $\mu$, i.e., 
\begin{equation}
q(\mathrm{Uni}(0,1))\sim\mu.\label{sampling}
\end{equation}
First, Gross expanded the function 
\[
\varphi:\begin{alignedat}{1}(-\pi,\pi) & \longrightarrow\mathbb{R}\\
\theta & \longmapsto q({\textstyle \frac{\vert\theta\vert}{\pi}})
\end{alignedat}
\]
as a Fourier series, i.e.
\begin{equation}
\varphi(\theta)=\sum^{+\infty}_{n=1}a_{n}\cos(n\theta).\label{fouier equality}
\end{equation}
Then he showed that the function
\[
G(z)=\sum^{+\infty}_{n=1}a_{n}z^{n},
\]
acting on the unit disc $\mathbb{D}$, is univalent. Finally he concluded
that $U:=G(\mathbb{D})$ solves the problem. Besides complex analysis
techniques, a core machinery of Gross' construction is the following
result. 
\begin{thm}[L\'evy's theorem]
 \label{thm:conformal}Let $f:U\rightarrow\mathbb{C}$ be a non constant
analytic function and $(Z_{t})_{t\geq0}$ be a planar Brownian motion
running inside $U$. Then there is a planar Brownian motion $W_{t}$
such that $f(Z_{t})=W_{\sigma(t)}$ with 
\[
\sigma(t)=\int^{t}_{0}\vert f'(Z_{s})\vert^{2}ds
\]
and $t\in[0,\tau_{U}]$. 
\end{thm}

The identity \eqref{fouier equality} holds both almost everywhere
and in $L^{p}$. This is the subject of Hunt-Carleson theorem \cite{arias2002pointwise,fefferman1973pointwise}.
As one can see, Gross' construction relies on the knowledge of the
Fourier coefficients of $\varphi$. These coefficients are difficult
to compute most of the time, especially when the c.d.f.\ is not explicit,
as in the normal distribution case. This motivates an approximation
perspective: when exact solutions are hard to obtain, we seek to approximate
(in a suitable sense) the underlying $\mu$-domain by a sequence of
domains. A natural idea is to consider probability measures $\mu_{n}\Longrightarrow\mu$,
i.e. convergence in law, and to study the behavior of the corresponding
$\mu_{n}$-domains.\\
\\
\\

We develop a numerical framework to approximate $\mu$-domains for
the PSEP; we analyze the convergence of $\mu_{n}$-domains under weak
convergence $\mu_{n}\Longrightarrow\mu$; we provide implementation
strategies, estimates of convergence rates, and numerical experiments
illustrating the method's robustness. We namely introduce a novel
mode of convergence for domains that underpins the analysis. Following
Gross' approach, let $G_{n}$ (resp.\ $G$) be the univalent map
generated by $\mu_{n}$ (resp.\ $\mu$); that is, 
\[
G_{n}(z)=\sum^{+\infty}_{n=1}a_{k,n}\,z^{k},
\]
\[
G(z)=\sum^{+\infty}_{n=1}a_{k}\,z^{k},
\]
where $a_{k,n}$ (resp.\ $a_{k}$) is the $k$-th Fourier coefficient
of $\varphi_{n}(\theta):=q_{n}(\frac{|\theta|}{\pi})$ (resp.\ $\varphi$).
Our aim is to show that $Z_{\tau_{n}}\to Z_{\tau}$ (hence the real
parts converge, as required) and $\tau_{n}\to\tau$ in some strong
way. Throughout the paper, we use $\mu$ as a probability distribution
with a finite $p$-th moment for some $p\ge1$ and $(\mu_{n})_{n}$
denotes a sequence of probability measures converging weakly to $\mu$
(additional assumptions will be stated). The random times $\tau$
(resp. $\tau_{n}$) denotes the exit time from the $\mu$-domain (resp.
$\mu_{n}$-domain). A subscript of the underlying domain is added
if needed. 

The following mathematical terms are adopted.
\begin{itemize}
\item $\mathbb{D}$ denotes the unit disc.
\item $\mathbb{S}^{1}$ denotes the unit circle and it is equipped with
the normalized Lebesgue measure. 
\end{itemize}
Our analysis suggests a broader perspective. The approximation of
the $\mu$-domains naturally leads to the introduction of a new notion
of convergence for planar domains, which we term \emph{$p$-Brownian
convergence} ($p>0$). 
\begin{defn}
\label{p-conv} Let $U$ be planar domain and let $(U_{n})_{n}$ be
a sequence of planar domains, all containing a common point, say the
origin. For each $n$, let $(Z^{[n]}_{t})_{t\geq0}$ be a planar Brownian
motion started at $0$ and denote by $\tau_{n}$ its exit time from
$U_{n}$. Similarly, let $(Z_{t})_{t\geq0}$ be a planar Brownian
motion started at $0$ and denote by $\tau$ its exit time from $U$.
We say that $(U_{n})_{n}$ converges to a domain $U$ in the \textbf{$p$-Brownian
sense} if there exists a coupling of the laws of $(Z^{[n]}_{\tau_{n}},\tau_{n})$
and $(Z_{\tau},\tau)$ on a common probability space such that 
\[
\mathscr{C}_{1}:\begin{alignedat}{1}\mathbf{E}\left(\big|Z^{[n]}_{\tau_{n}}-Z_{\tau}\big|^{p}\right) & \underset{n\to+\infty}{\longrightarrow}0\end{alignedat}
\]
and 
\[
\mathscr{C}_{2}:\begin{alignedat}{1}\mathbf{E}\left(\big|\tau_{n}-\tau\big|^{\frac{p}{2}}\right) & \underset{n\to+\infty}{\longrightarrow}0\end{alignedat}
.
\]

This definition captures the geometric proximity of the domains in
terms of Brownian trajectories. More discussion of the \emph{$p$-Brownian
convergence }is presented in the last section. The main result of
this work is\emph{ }
\end{defn}

\begin{thm}[Stability of the PSEP]
\label{stability} If $G_{n}\underset{n\to\infty}{\longrightarrow}G$
in the Hardy space $\mathbf{H}^{p}(\mathbb{D})$ for some $p\geq1$
then the sequence of $\mu_{n}$-domains converges to the target $\mu$-domain
in the $p$-Brownian sense. 
\end{thm}

The feasibility of the condition $G_{n}\underset{n\to\infty}{\longrightarrow}G$
in $\mathbf{H}^{p}(\mathbb{D})$ and more will be discussed in the
proof of theorem \ref{stability}. 

\section{Tools and proofs}

We begin by recalling the mathematical tools required to show our
main result. 
\begin{defn}
We say that a sequence of probability distributions $(\mu_{n})$ converges
weakly to $\mu$ if the distribution functions $F_{n}(x):=\mu_{n}((-\infty,x])$
converge to $F(x):=\mu((-\infty,x])$ at every continuity point of
$F$. 
\end{defn}

Weak convergence is strictly weaker than setwise convergence. For
further details on convergence of measures, see \cite{billingsley2013convergence}.
The following fact can be found in \cite{resnick2008extreme}. It
is a key ingredient of our analysis. 
\begin{thm}
\label{thm:clue thm}The sequence of quantile functions $q_{n}$ converges
to $q$ almost everywhere on $(0,1)$. 
\end{thm}

\begin{defn}
The Hilbert transform of a real valued $2\pi$- periodic function
$f$ is defined by 
\[
H\{f\}(x):=PV\left\{ \frac{1}{2\pi}\int^{\pi}_{-\pi}f(x-t)\cot(\frac{t}{2})dt\right\} =\lim_{\eta\rightarrow0}\frac{1}{2\pi}\int_{\eta\leq|t|\leq\pi}f(x-t)\cot(\frac{t}{2})dt
\]
where $PV$ denotes the Cauchy principal value. The Hilbert transform
does exist a.e whenever $p\geq1$. However, it is a bounded operator
on $L^{p}_{2\pi}$ only when $p>1$. The case $p=1$ is irregular,
and $H$ becomes unbounded. For all proofs and properties of the Hilbert
transform, we refer the reader to \cite{butzer1971hilbert,king2009hilbert}. 
\end{defn}

\begin{defn}
Let $f$ be analytic on the unit disc $\mathbb{D}$ and $p>0$. The
$p$-th Hardy norm of $f$ is 
\begin{equation}
\Vert f\Vert_{\mathbf{H}^{p}(\mathbb{D})}:=\sup_{0\le r<1}\Big\{\frac{1}{2\pi}\int^{2\pi}_{0}\big|f(re^{i\theta})\big|^{p}\,d\theta\Big\}^{\frac{1}{p}}.\label{hardy}
\end{equation}
The supremum is well-defined\cite{Rudin2001}. The Hardy space, denoted
by $\mathbf{H}^{p}(\mathbb{D})$, consists of analytic $f$ with $\Vert f\Vert_{\mathbf{H}^{p}(\mathbb{D})}<+\infty$.
A key fact of Hardy spaces is the boundary trace: if $\Vert f\Vert_{\mathbf{H}^{p}(\mathbb{D})}<+\infty$,
then the radial boundary values $f^{*}(e^{i\theta}):=\lim_{r\to1}f(re^{i\theta})$
exists a.e.\ on $\mathbb{S}$. Moreover 
\[
\|f^{*}\|_{L^{p}(\mathbb{S})}=\Vert f\Vert_{\mathbf{H}^{p}(\mathbb{D})}.
\]
When $p>1$, $\mathbf{H}^{p}(\mathbb{D})$ identifies with $L^{p}_{2\pi}$
in the following way: if $\xi\in L^{p}_{2\pi}$ then the map
\begin{equation}
\Psi(re^{\theta i})=\frac{1}{2\pi}\int^{2\pi}_{0}\frac{1+re^{(\theta-t)i}}{1-re^{(\theta-t)i}}\xi(t)dt\label{Riesz-Herglotz Representation}
\end{equation}
belongs to $\mathbf{H}^{p}(\mathbb{D})$ and $\Re(\Psi^{*})=\xi$.
When $p=1$, $\Psi$ is analytic in the unit disc but not necessarily
an element of $\mathbf{H}^{1}(\mathbb{D})$. We recommend \cite{duren2000theory,butzer1971hilbert,king2009hilbert,mcgovern1980hilbert,pandey2011hilbert}
for a concise background on Hardy spaces as well as their connections
with the Hilbert transform. The following result illustrates a deep
connection between Hardy spaces and planar Brownian motion. 
\end{defn}

\begin{thm}
\label{boundary dist}\cite{boudabra2019remarks,BOUDABRA2026} Let
$f$ be a univalent function in the unit disc with $f(0)=0$, and
set $U=f(\mathbb{D})$. Let $(Z_{t})_{t\geq0}$ be a planar Brownian
motion and $\tau$ be its exit time from $U$. Then 
\[
Z_{\tau}\sim f^{*}(\xi)
\]
where $\xi$ is any random variable uniformly distributed in $\mathbb{S}^{1}$.
In particular 
\[
\mathbf{E}\left(\vert Z_{\tau}\big|^{p}\right)=\Vert f\Vert^{p}_{\mathbf{H}^{p}(\mathbb{D})}.
\]
 
\end{thm}

\subsection{Proof of theorem \ref{stability}.}

In order to prove theorem \ref{stability}, which concerns the $\mu_{n}$-domains
$U_{n}:=G_{n}(\mathbb{D})$ and the target $\mu$-domain $U=G(\mathbb{D})$,
we use the canonical coupling induced by a common exit point of Brownian
motion from the unit disc $\mathbb{D}$.
\begin{proof}
By conformal invariance principle stated in theorem \ref{thm:conformal},
the image of planar Brownian motion under a nonconstant analytic map
is Brownian motion with a time-change \cite{revuz2013continuous}.
That is, let $(B_{t})_{0\leq t\leq\tau_{\mathbb{D}}}$ be a planar
Brownian motion running inside the unit disc (killed on the unit circle).
By Lévy\textquoteright s theorem \ref{thm:conformal}, there exist
Brownian motions $(Z^{[n]}_{t})_{t}$ in $U_{n}$ and $(Z_{t})_{t}$
in $U$ such that 
\[
Z^{[n]}_{\sigma_{n}(t)}=G_{n}(B_{t})_{0\leq t\leq\tau_{\mathbb{D}}}
\]
and 
\[
Z_{\sigma(t)}=G(B_{t})_{0\leq t\leq\tau_{\mathbb{D}}}
\]
are two time changed planar Brownian paths running inside $U_{n}$
and $U$, where $\sigma_{n}$ and $\sigma$ are the time changes as
in theorem \ref{thm:conformal}, i.e. 
\[
\sigma(t)=\int^{t}_{0}\big|G'(B_{s})\big|^{2}ds,\qquad\sigma_{n}(t)=\int^{t}_{0}\big|G_{n}'(B_{s})\big|^{2}ds.
\]
This construction makes the processes coupled. In particular, 
\[
\tau=\sigma(\tau_{\mathbb{D}}),\qquad\tau_{n}=\sigma_{n}(\tau_{\mathbb{D}}).
\]
Equivalently, $\tau$ and $\tau_{n}$ are the quadratic variations
of the martingales $G(B_{s})$ and $G_{n}(B_{s})$ up to time $\tau_{\mathbb{D}}$.
Let $\theta\in(-\pi,\pi)$ be defined by $e^{\theta i}:=B_{\tau_{\mathbb{D}}}$.
Then $\theta$ uniformly distributed in $(-\pi,\pi)$ since $(B_{t})_{t}$
starts at $0$. By theorem \ref{boundary dist} we get
\[
Z_{\tau}=G^{*}(e^{\theta i}),\,\,\,\,Z^{[n]}_{\tau_{n}}=G^{*}_{n}(e^{\theta i}).
\]
Hence 
\[
\mathbf{E}\left(\big|Z^{[n]}_{\tau_{n}}-Z_{\tau}\big|^{p}\right)=\int^{\pi}_{-\pi}\big|G^{*}_{n}(e^{\theta i})-G^{*}(e^{\theta i})\big|^{p}\frac{d\theta}{2\pi}=\|G^{*}_{n}-G^{*}\|^{p}_{L^{p}(\mathbb{S}^{1})}=\Vert G_{n}-G\Vert^{p}_{\mathbf{H}^{p}(\mathbb{D})}\underset{n\to\infty}{\longrightarrow}0.
\]
Thus, $\mathscr{C}_{1}$ is established. For $\mathscr{C}_{2}$, write
$\Delta_{n}:=G_{n}-G$ .Then 
\[
|\tau_{n}-\tau|=\bigg|\int^{\tau_{\mathbb{D}}}_{0}\Big(|G_{n}'|^{2}-|G'|^{2}\Big)(B)ds\bigg|\le V_{n}+2\int^{\tau_{\mathbb{D}}}_{0}|G'(B_{s})||\Delta_{n}'(B_{s})|\,ds
\]
with 
\[
V_{n}=\int^{\tau_{\mathbb{D}}}_{0}|\Delta_{n}'(B_{s})|^{2}ds.
\]
By the Cauchy--Bunyakovsky--Schwarz inequality, 
\[
\int^{\tau_{\mathbb{D}}}_{0}|G'(B_{s})||\Delta_{n}'(B_{s})|ds\le\tau^{\frac{1}{2}}V_{n}{}^{\frac{1}{2}}.
\]
Hence, with $\delta=\frac{p}{2}>0$ and for some constant $\kappa_{\delta}>0$,
\[
|\tau_{n}-\tau|^{\delta}\le\kappa_{\delta}\left(V^{\delta}_{n}+\tau^{\frac{\delta}{2}}V^{\frac{\delta}{2}}_{n}\right).
\]
Taking expectations and another round of the Cauchy--Bunyakovsky--Schwarz
yields 
\[
\mathbf{E}\left(|\tau_{n}-\tau|^{\delta}\right)\le\kappa_{\delta}\left(\mathbf{E}\left(V^{\delta}_{n}\right)+\mathbf{E}(\tau^{\delta})^{\frac{1}{2}}\mathbf{E}\left(V^{\delta}_{n}\right)^{\frac{1}{2}}\right).
\]
The famous Burkholder-Davis-Grundy inequality for analytic martingales
\cite{BurkholderDavisGundy1972} \cite{burkholder1977exit}, guarantees
the existence of constants $0<c_{p}\le C_{p}<\infty$ (independent
of $n$) such that for any analytic function $h$ with $h(0)=0$,
\begin{equation}
c_{p}\Vert h\Vert^{p}_{\mathbf{H}^{p}(\mathbb{D})}\le\mathbf{E}\left(\int^{\tau_{\mathbb{D}}}_{0}|h'(B_{s})|^{2}ds\right)^{\delta}\le C_{p}\Vert h\Vert^{p}_{\mathbf{H}^{p}(\mathbb{D})}.\label{bdg}
\end{equation}
Applying this with $h=\Delta_{n}$ and with $h=G$ gives 
\[
\mathbf{E}\left(V^{\delta}_{n}\right)\lesssim\|\Delta^{*}_{n}\|^{p}_{L^{p}(\mathbb{S})}\underset{n\to\infty}{\longrightarrow}0,\qquad\mathbf{E}(\tau^{\delta})\lesssim\|G^{*}\|^{p}_{L^{p}(\mathbb{S})}<\infty,
\]
which implies $\mathbf{E}\left(|\tau_{n}-\tau|^{\delta}\right)\underset{n\to\infty}{\longrightarrow}0$,
i.e., $\tau_{n}\underset{n\to\infty}{\longrightarrow}\tau$ in $L^{\delta}$.
That is, $\mathscr{C}_{2}$ is proved.
\end{proof}

\begin{rem}
Theorem \ref{stability} is genuinely a stability result in the Hardy
topology. For example, local uniform convergence of the conformal
maps is not sufficient to guarantee $p$-Brownian convergence. Indeed,
let

\[
U:=\{z\in\mathbb{C}\mid\ \Re(z)>-1\},\qquad U_{n}:=\{z\in\mathbb{C}\mid\vert z-n\vert<n+1\},\qquad n\ge1.
\]
Then each $U_{n}$ is simply connected, $0\in U_{n}$, and
\[
U_{n}\uparrow U.
\]
A normalized conformal map $G_{n}:\mathbb{D}\to U_{n}$ is given by
\[
G_{n}(z)=\frac{(2n+1)z}{(n+1)-nz},
\]
while the limit map $G:\mathbb{D}\to U$ is

\[
G(z)=\frac{2z}{1-z}.
\]
It is immediate that $G_{n}(0)=G(0)=0$, and one checks that
\[
G_{n}\longrightarrow G\qquad\text{locally uniformly on }\mathbb{D}.
\]
However, $G\notin\mathbf{H}^{p}(\mathbb{D})$ for any $p\ge1$. Indeed,
\[
G^{*}(e^{i\theta})=\frac{2e^{i\theta}}{1-e^{i\theta}},
\]
and hence $|G^{*}(e^{i\theta})|^{p}$ is not integrable near $0$
whenever $p\ge1$. Let $\tau_{n}$ and $\tau$ denote the exit times
from $U_{n}$ and $U$, respectively, for Brownian motion started
at $0$. Since
\[
U=\{z\in\mathbb{C}:\ \Re(z)>-1\},
\]
the exit time $\tau$ is exactly the hitting time of $-1$ by a one-dimensional
Brownian motion. Therefore
\[
\mathbb{\mathbf{E}}(\tau^{\delta})<+\infty\quad\Longleftrightarrow\quad\delta<\tfrac{1}{2}.
\]
In particular,
\[
\mathbb{\mathbf{E}}(\tau^{\frac{p}{2}})=+\infty\qquad\text{for every }p\ge1.
\]
On the other hand, each $U_{n}$ is bounded, hence
\[
\mathbb{\mathbf{E}}(\tau^{\frac{p}{2}}_{n})<+\infty.
\]
We claim that for every coupling of $\tau_{n}$ and $\tau$ one has
\[
\mathbb{\mathbf{E}}(|\tau_{n}-\tau|^{\frac{p}{2}})=+\infty.
\]
Set $\delta:=\frac{p}{2}$. If $\frac{1}{2}<\delta\le1$, then for
all $a,b\ge0$,
\[
a^{\delta}\le|a-b|^{\delta}+b^{\delta}.
\]
Applying this with $a=\tau$ and $b=\tau_{n}$, and taking expectations,
gives
\[
\mathbb{\mathbf{E}}(\tau^{\delta})\le\mathbb{\mathbf{E}}(|\tau-\tau_{n}|^{\delta})+\mathbb{\mathbf{E}}(\tau^{\delta}_{n}).
\]
Since $\mathbb{\mathbf{E}}(\tau^{\delta})=+\infty$ and $\mathbb{\mathbf{E}}(\tau^{\delta}_{n})<\infty$,
it follows that
\[
\mathbb{\mathbf{E}}(|\tau-\tau_{n}|^{\delta})=+\infty.
\]
If $\delta\ge1$, then Minkowski's inequality gives
\[
\mathbb{\mathbf{E}}(\tau^{\delta})^{\frac{1}{\delta}}\le\mathbb{\mathbf{E}}(|\tau-\tau_{n}|^{\delta})^{\frac{1}{\delta}}+\mathbb{\mathbf{E}}(\tau^{\delta}_{n})^{\frac{1}{\delta}}.
\]
Again, since $\mathbb{\mathbf{E}}(\tau^{\delta})^{\frac{1}{\delta}}=+\infty$
while $\mathbb{\mathbf{E}}(\tau^{\delta}_{n})^{\frac{1}{\delta}}<+\infty$,
we obtain
\[
\mathbb{\mathbf{E}}(|\tau-\tau_{n}|^{\delta})^{\frac{1}{\delta}}=+\infty,
\]
that is,
\[
\mathbb{\mathbf{E}}(|\tau-\tau_{n}|^{\delta})=+\infty.
\]
Thus, for every $p\ge1$, the sequence $(U_{n})_{n}$ does not converge
to $U$ in the $p$-Brownian sense, although the conformal maps $G_{n}$
converge locally uniformly to $G$. This shows that local uniform
convergence of the conformal maps is not enough; the $\mathbf{H}^{p}(\mathbb{D})$
assumption in Theorem \ref{stability} is essential.
\end{rem}

\begin{prop}
\label{prop} If $\varphi_{n}\to\varphi$ in $L^{p}_{2\pi}$ with
$p>1$, then $G_{n}\underset{n\to\infty}{\longrightarrow}G$ in $\mathbf{H}^{p}(\mathbb{D})$. 
\end{prop}

The proof of Proposition \ref{prop} can be found in \cite{duren2000theory} for
example. The conclusion may fail for $p=1$ because Hilbert transform
is unbounded on $L^{1}_{2\pi}$, but if we add the condition 
\[
H\{\varphi_{n}\}\underset{n\to\infty}{\longrightarrow}H\{\varphi\}\,\,\text{in}\,L^{1}_{2\pi}
\]
then $G_{n}\underset{n\to\infty}{\longrightarrow}G$ in $\mathbf{H}^{1}(\mathbb{D})$.
For further details about the PSEP when $p=1$, see \cite{BOUDABRA2026}.
Using the monotonicity of the quantiles, we give a convergence criterion
to get the condition $\varphi_{n}\to\varphi$ in $L^{p}_{2\pi}$ .
\begin{prop}
\label{proposition} Let $p>0$. If for some $0<\delta<\tfrac{1}{2}$,
\begin{equation}
\max\left(\int^{\delta}_{0}|q_{n}-q|^{p}\,du,\int^{1}_{1-\delta}|q_{n}-q|^{p}\,du\right)\underset{n\to\infty}{\longrightarrow}0,\label{delta integral-1}
\end{equation}
then $\varphi_{n}\underset{n\to\infty}{\longrightarrow}\varphi$ in
$L^{p}_{2\pi}$ .
\end{prop}

\begin{proof}
Let $S\subset(0,1)$ be the set where $q_{n}\underset{n\to\infty}{\longrightarrow}q$
a.e; then $\lambda(S)=1$. Set $S_{\delta}=[\delta,1-\delta]\cap S$.
Since each $q_{n}$ is nondecreasing and $q_{n}\underset{n\to\infty}{\longrightarrow}q$
a.e., for any $\varepsilon>0$ there exists $N$ with 
\[
\inf_{u\in S_{\delta}}q(u)-\varepsilon\le q_{n}(u)\le\sup_{u\in S_{\delta}}q(u)+\varepsilon\qquad(u\in S_{\delta},~n\ge N),
\]
so $(q_{n})_{n}$ is bounded on $S_{\delta}$. By dominated convergence,
\begin{equation}
\int_{S_{\delta}}|q_{n}-q|^{p}du=\int^{1-\delta}_{\delta}|q_{n}-q|^{p}du\underset{n\to\infty}{\longrightarrow}0.\label{S_delta-1}
\end{equation}
Combining \eqref{S_delta-1} with \eqref{delta integral-1} yields
\[
\|q_{n}-q\|_{L^{p}(0,1)}=\|\varphi_{n}-\varphi\|_{L^{p}_{2\pi}}\underset{n\to\infty}{\longrightarrow}0.
\]
\end{proof}

\noindent The assumption \eqref{delta integral-1} is obviously a
necessary condition. It holds, for instance, when the sequence $(\mu_{n})_{n}$
is of uniform bounded support. 
\begin{rem}[Wasserstein sandwich]
 The $p$-Wasserstein distance between two measures is defined by
\[
W_{p}(\mu,\nu):=\Big(\inf_{\pi\in\Pi(\mu,\nu)}\int|x-y|^{p}\,d\pi(x,y)\Big)^{\frac{1}{p}}.
\]
Let $\omega_{U_{n}}(0,\cdot)$ and $\omega_{U}(0,\cdot)$ denote the
harmonic measures seen from the origin in $U_{n}:=G_{n}(\mathbb{D})$
and $U:=G(\mathbb{D})$, and let $\mu_{n},\mu$ be their real-part
pushforwards. Since $\Re:\mathbb{C}\to\mathbb{R}$ is $1$-Lipschitz,
Wasserstein distances contract under pushforward, hence
\[
W_{p}(\mu_{n},\mu)^{p}\le W_{p}\big(\omega_{U_{n}}(0,\cdot),\omega_{U}(0,\cdot)\big)^{p}.
\]
By the definition of $W_{p}$ and taking the canonical coupling from
the proof of theorem \ref{stability}, we get a coupling cost 
\[
W_{p}\big(\omega_{U_{n}}(0,\cdot),\omega_{U}(0,\cdot)\big)^{p}\le\mathbf{E}(|Z^{[n]}_{\tau_{n}}-Z_{\tau}|^{p})=\Vert G_{n}-G\Vert^{p}_{\mathbf{H}^{p}(\mathbb{D})}.
\]
Finally, in one dimension
\begin{equation}
W_{p}(\mu_{n},\mu)^{p}=\|q_{n}-q\|^{p}_{L^{p}(0,1)}=\|\varphi_{n}-\varphi\|^{p}_{L^{p}(\mathbb{S}^{1})}\label{W and =00005Cvarphi}
\end{equation}
 (see \cite{Ruschendorf1985}). Therefore,
\[
\|\varphi_{n}-\varphi\|_{L^{p}(\mathbb{S}^{1})}=W_{p}(\mu_{n},\mu)\le W_{p}\big(\omega_{U_{n}}(0,\cdot),\omega_{U}(0,\cdot)\big)\le\mathbf{E}(|Z^{[n]}_{\tau_{n}}-Z_{\tau}|^{p})^{\frac{1}{p}}=\Vert G_{n}-G\Vert_{\mathbf{H}^{p}(\mathbb{D})},
\]
which reads as
\[
W_{p}\big(\omega_{U_{n}}(0,\cdot),\omega_{U}(0,\cdot)\big)\in[\|\varphi_{n}-\varphi\|_{L^{p}(\mathbb{S}^{1})},\Vert G_{n}-G\Vert_{\mathbf{H}^{p}(\mathbb{D})}].
\]
\end{rem}

\section{Implementation}

\noindent In this section we provide an implementation framework for
the results established above. Let $\mu$ be a probability measure
with bounded connected support $[a,b]$, possibly with a finite number
of atoms $a_{1}<\dots<a_{s}$. Its quantile function is 
\begin{equation}
q(u)=\sum^{s+1}_{i=1}F^{-1}(u)\,1_{\{u\in(F(a_{i-1}),F(a^{-}_{i}))\}}+\sum^{s}_{k=1}a_{k}\,1_{\{u\in(F(a^{-}_{k}),F(a_{k}))\}},\label{q(u)}
\end{equation}
with the conventions $a_{0}=a$, $a_{s+1}=b$, and $\sum_{\emptyset}=0$.
The first term in the right-hand side of \eqref{q(u)} represents
the continuous part of $\mu$, while the second encodes the discrete
atomic masses. The following elementary result provides a natural
sequence of probability measures $\mu_{n}$ converging weakly to $\mu$.
A mesh $(x^{[n]}_{k})_{0\leq k\leq n}$ of the interval $[a,b]$ is
any ordered family of points 
\[
a=x_{0}<\cdots<x_{n}=b.
\]

\begin{lem}
Let $(x^{[n]}_{k})_{0\leq k\leq n}$ be a mesh such that $\max_{1\le k\le n}(x_{k}-x_{k-1})\underset{n\rightarrow+\infty}{\longrightarrow}0$.
Then the sequence of probability measures 
\[
\mu_{n}=F(a)\delta_{a}+\sum^{n}_{k=1}\big(F(x_{k})-F(x_{k-1})\big)\,\delta_{x_{k}}
\]
converges weakly to $\mu$. 
\end{lem}

This discretization allocates the total $\mu$-mass of each interval
$(x_{k-1},x_{k}]$ to the grid point $x_{k}$ (and similarly assigns
the mass $F(a)$ to $a$); in particular, atoms are not treated separately.
Consequently, an atom at $a_{0}\in(x_{k-1},x_{k}]$ is represented
at $x_{k}$ for finite $n$, but this snapping error vanishes as the
mesh size $\max_{k}(x_{k}-x_{k-1})\to0$. Since $q_{n}$ is bounded,
it converges to $q$ in every $L^{p}$ with $p\geq1$. Our next result
quantifies the convergence rate. 
\begin{prop}
\label{prop:rate} Let $(x^{[n]}_{k})_{0\leq k\leq n}$ be a mesh.
Define
\[
\mu_{n}:=F(a)\delta_{a}+\sum^{n}_{k=1}\big(F(x_{k})-F(x_{k-1})\big)\delta_{x_{k}},
\]
and let $q_{n}$ be the quantile function of $\mu_{n}$. Then for
any $p\ge1$,
\[
\|q-q_{n}\|_{L^{p}(0,1)}\le\max_{1\le k\le n}(x_{k}-x_{k-1}).
\]
In particular, for the uniform mesh $x_{k}=a+\frac{b-a}{n}k$,
\[
\|q-q_{n}\|_{L^{p}(0,1)}\le\frac{b-a}{n}.
\]
\end{prop}

\begin{proof}
By construction, the c.d.f $F_{n}$ of $\mu_{n}$ satisfies $F_{n}(x_{k})=F(x_{k})$
for each $k$. Hence $q_{n}(u)=x_{k}$ for $u\in(F(x_{k-1}),F(x_{k})]$.
For such $u$ we have $F(x_{k-1})<u\le F(x_{k})$, so by the definition
of the quantile $q(u)=\inf\{x:\,F(x)\ge u\}$ lies in $[x_{k-1},x_{k}]$.
Therefore 
\[
|q(u)-q_{n}(u)|\le x_{k}-x_{k-1}
\]
 on $(F(x_{k-1}),F(x_{k})]$. Taking the $L^{p}$ norm over $(0,1)$
yields 
\[
\|q-q_{n}\|_{L^{p}(0,1)}\le\max_{1\le k\le n}(x_{k}-x_{k-1}).
\]
\end{proof}

\begin{rem}[Sharpness]
For $\mu=\mathrm{Unif}\left(0,1\right)$ and the uniform mesh, we
have $q(u)=u$ and, with the right-endpoint scheme, $q_{n}(u)=\frac{k}{n}$
on $u\in(\frac{k-1}{n},\frac{k}{n}]$. Then 
\[
\|q-q_{n}\|^{p}_{L^{p}(0,1)}=\sum^{n}_{k=1}\int^{\frac{k}{n}}_{\frac{k-1}{n}}\left(\tfrac{k}{n}-u\right)^{p}\,du=n\int^{1/n}_{0}t^{p}\,dt=\frac{1}{(p+1)\,n^{p}},
\]
so 
\[
\|q-q_{n}\|_{L^{p}(0,1)}=\frac{1}{(1+p)^{\frac{1}{p}}n},
\]
showing the $\frac{1}{n}$ rate is optimal. 
\end{rem}

Using the boundedness of the Hilbert transform for $p>1$, we then
obtain 
\begin{cor}
If the measure $\mu$ is of bounded support and the mesh is uniform
then the estimate 
\[
\Vert G_{n}-G\Vert_{\mathbf{H}^{p}(\mathbb{D})}=O(\frac{1}{n})
\]
holds for $p>1$.
\end{cor}

Based on the estimate 
\[
|\tau_{n}-\tau|^{\delta}\le\kappa_{\delta}\left(V^{\delta}_{n}+\tau^{\frac{\delta}{2}}V^{\frac{\delta}{2}}_{n}\right)
\]
appeared in the proof theorem \ref{stability}, $\tau_{n}$ still
converges to $\tau$ in $L^{\delta}$ with rate $\frac{1}{n}$.When
$p=1$ then the rate of convergence of $\Vert G-G_{n}\Vert_{\mathbf{H}^{1}(\mathbb{D})}$
cannot be deduced from that of $\|q-q_{n}\|_{L^{1}(0,1)}$ without
further control on the imaginary part of $G_{n}-G$. The parametrization
of the boundary of $\mu$-domains obtained by Gross' technique is
\[
\theta\in(-\pi,\pi)\longmapsto\big(\Re(G(e^{\theta i})),\Im(G(e^{\theta i}))\big)=\big(\varphi(\theta),\,H\{\varphi\}(\theta)\big),
\]
see \cite{Boudabra2020,gross2019}. For step functions, we require
the transform of indicators. A direct calculation gives 
\begin{equation}
H\{\mathbf{1}_{\{|\cdot|\in(\alpha,\beta)\}}\}(\theta)=\frac{1}{\pi}\ln\left(\frac{\sin\big(\tfrac{\theta-\alpha}{2}\big)\sin\big(\tfrac{\theta+\beta}{2}\big)}{\sin\big(\tfrac{\theta-\beta}{2}\big)\sin\big(\tfrac{\theta+\alpha}{2}\big)}\right).\label{hilbert}
\end{equation}

Thus the Hilbert transform of $\varphi_{n}$ is 
\[
\begin{aligned}H\{\varphi_{n}\}(\theta) & =\sum^{n}_{k=0}x_{k}H\left\{ \mathbf{1}_{\{|\cdot|\in(\theta_{k-1},\theta_{k}]\}}\right\} (\theta)\\
 & =\frac{1}{\pi}\sum^{n}_{k=0}x_{k}\ln\left(\frac{\sin\big(\tfrac{\theta-\theta_{k-1}}{2}\big)\sin\big(\tfrac{\theta+\theta_{k}}{2}\big)}{\sin\big(\tfrac{\theta-\theta_{k}}{2}\big)\sin\big(\tfrac{\theta+\theta_{k-1}}{2}\big)}\right),\qquad\theta\in(-\pi,\pi)
\end{aligned}
\]
where $\theta_{k}=\pi F(x_{k})$ subject to the convention $\theta_{-1}=0$. 

\begin{figure}[H]
\centering \includegraphics[width=8cm,totalheight=8cm,keepaspectratio]{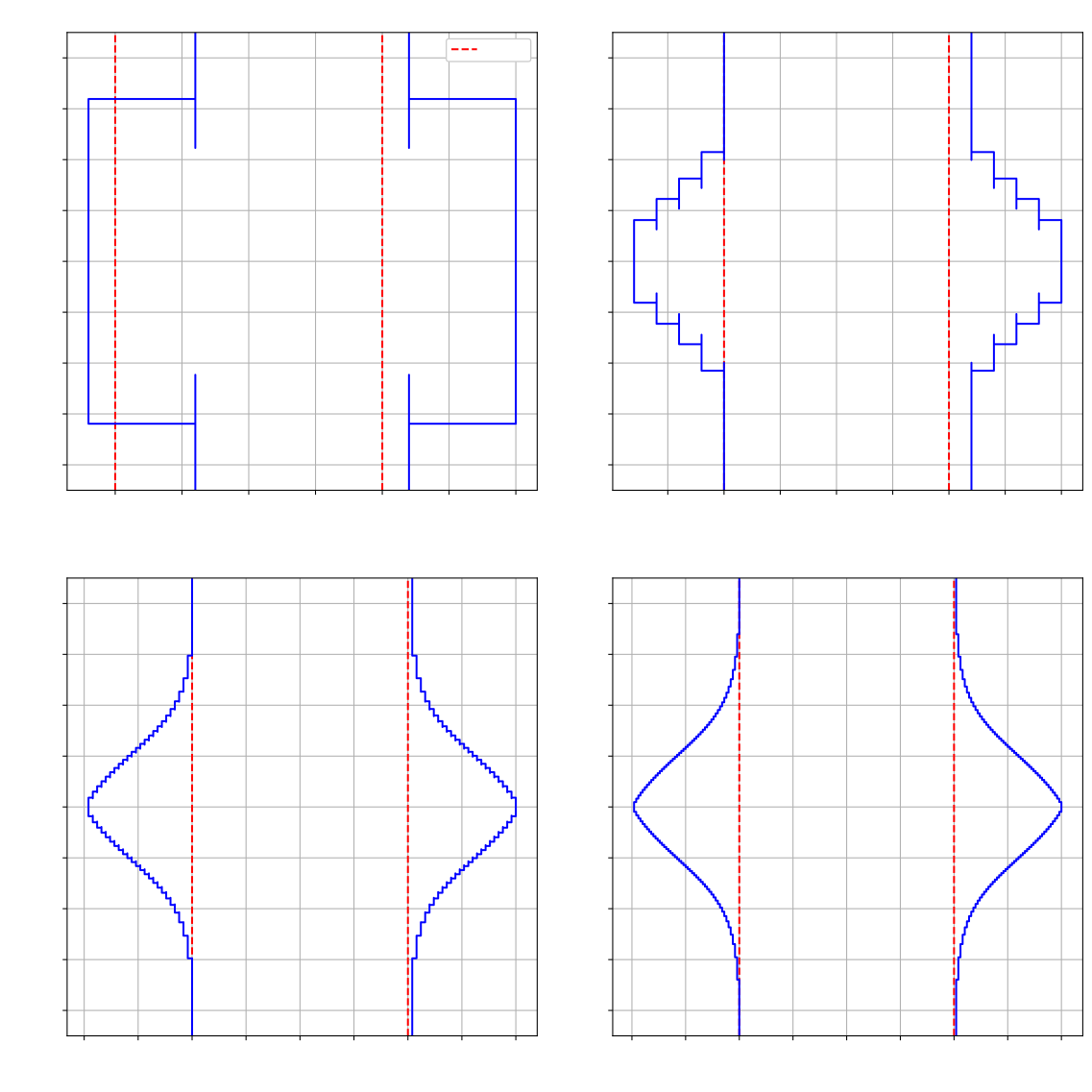}
\caption{Approximation of the $\mu$-domain generated from $\mu=\mathrm{Uni}((-2,-1)\cup(1,2))$
for $n=5,20,100,200$. The final $\mu$-domain contains the vertical
strip $\{-1<x<1\}$. }
\end{figure}

\begin{figure}[H]
\centering \includegraphics[width=8cm,totalheight=8cm,keepaspectratio]{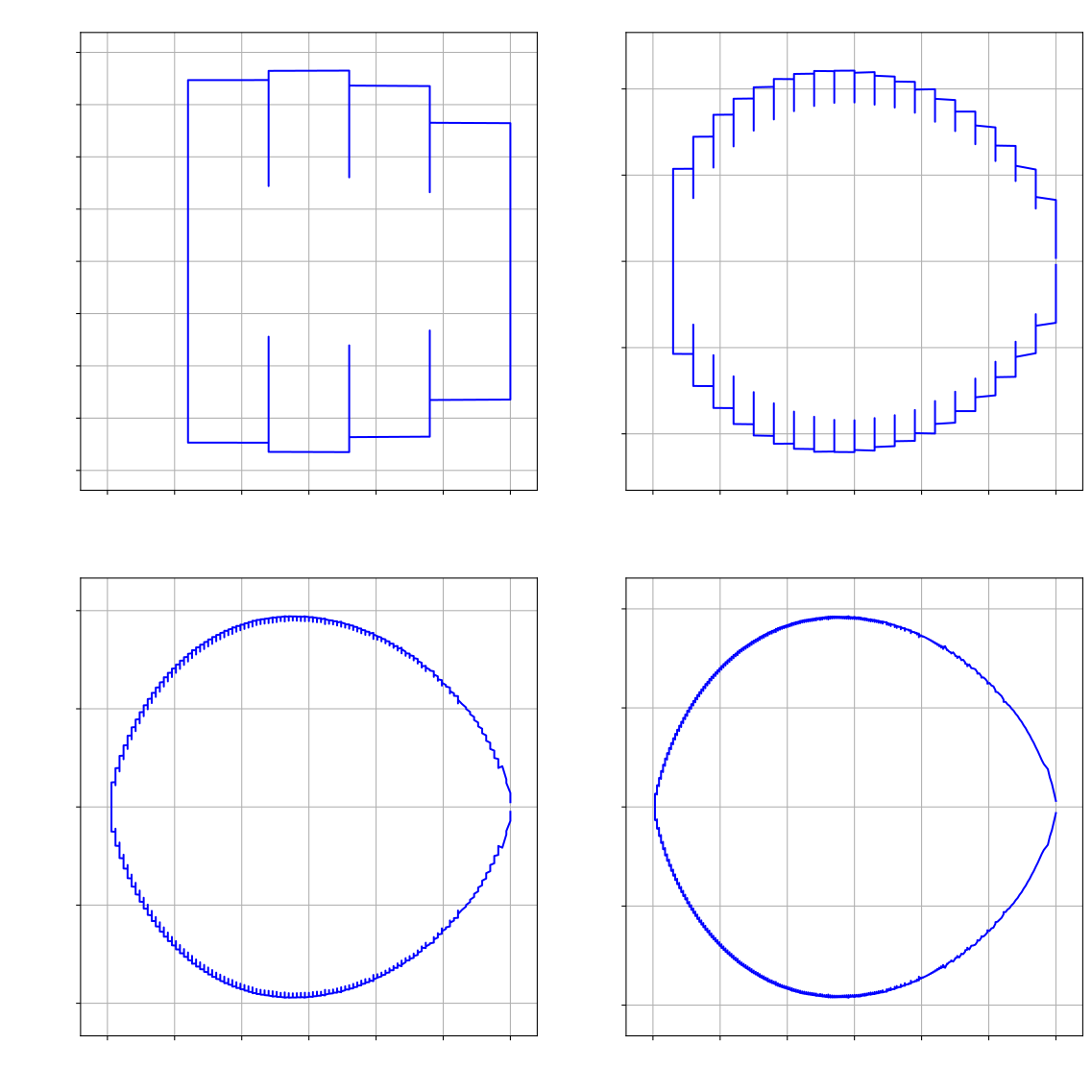}
\caption{Approximation of the $\mu$-domain generated from the truncated exponential
distribution $\mathrm{Exp}(1)$ on $(0,3)$ for $n=5,20,100,200$. }
\end{figure}

\begin{rem} For unbounded distributions, or for distributions with
very large effective support, a uniform discretization on an interval
$[a,b]$ may require very large $n$ before the mesh size $\frac{b-a}{n}$
becomes numerically useful. In such cases, one may first rescale the
measure and work on a normalized domain size, and then recover the
original $\mu$-domain by the scaling properties of the construction.
Note tat this is a numerical normalization device and does not improve
the theoretical order of convergence.
\end{rem}

\section{Comments}

In this work we proposed a numerical framework to approximate the
$\mu$-domain associated with a given distribution. The assumption
on the sequence $(\mu_{n})_{n}$ is among the weakest available in
the literature, which highlights the generality of the approach. Our
theoretical results were complemented with explicit constructions
and convergence rates, as well as practical implementation strategies
based on the Hilbert transform. Numerical simulations demonstrated
the robustness of the method, thereby validating both the effectiveness
and the versatility of the scheme. It is worth noting that the framework
is compatible not only with Gross' original construction, but also
with the domains introduced by Boudabra and Markowsky in \cite{Boudabra2020},
showing that the method is adaptable to distinct formulations of the
planar Skorokhod embedding problem.

\medskip{}

\subsubsection*{Comment 1. }

Theorem \ref{stability} is quantitative: it controls the Brownian
outputs in terms of the Hardy distance between the underlying conformal
maps. Indeed,

\[
\mathbf{E}(|Z^{[n]}_{\tau_{n}}-Z_{\tau}|^{p})=\Vert G_{n}-G\Vert^{p}_{\mathbf{H}^{p}(\mathbb{D})},
\]
and the proof yields

\[
\mathbf{E}(|\tau_{n}-\tau|^{\frac{p}{2}})\le\eta_{p}\Big(\Vert G_{n}-G\Vert^{p}_{\mathbf{H}^{p}(\mathbb{D})}+\Vert G_{n}-G\Vert^{\frac{p}{2}}_{\mathbf{H}^{p}(\mathbb{D})}\Big),
\]
where $\eta_{p}$ depends only on $p$ and $\mathbf{E}(\tau^{\frac{p}{2}})$
. If $f$ is analytic in the unit disc then 
\[
\Vert f\Vert_{\mathbf{H}^{p}(\mathbb{D})}\lesssim\Vert\Re(f^{*})\Vert_{L^{p}(\mathbb{S})}
\]
provided that $p>1$ \cite{duren2000theory}. In particular
\[
\Vert G_{n}-G\Vert^{p}_{\mathbf{H}^{p}(\mathbb{D})}\lesssim\Vert\varphi-\varphi_{n}\Vert_{L^{p}(\mathbb{S})}.
\]
Thus, in Gross' construction, via \ref{W and =00005Cvarphi} in remark
$12$, one obtains the estimate
\[
\mathbf{E}(|\tau_{n}-\tau|^{\frac{p}{2}})\lesssim W_{p}(\mu_{n},\mu)^{p}+W_{p}(\mu_{n},\mu)^{\frac{p}{2}}
\]
Hence our bounds can be rephrased in terms of a transport metric $d(\mu_{n},\mu)=W_{p}(\mu_{n},\mu)$.
Comparing this dependence for other PSEP constructions and other transport
metrics is an interesting direction to investigate.

\medskip{}

\subsubsection*{Comment 2. }

One may wonder why not to consider just one planar Brownian motion
$(Z_{t})_{t\geq0}$ and use the two conditions 
\[
\mathbf{E}\left(\big|Z_{\tau_{n}}-Z_{\tau}\big|^{p}\right)\underset{n\to+\infty}{\longrightarrow}0,\,\mathbf{E}\left(\big|\tau_{n}-\tau\big|^{\frac{p}{2}}\right)\underset{n\to+\infty}{\longrightarrow}0
\]
instead of asking for the existence of a convenient coupling. This
question is legitimate and expected. Using the same planar Brownian
motion in $\mathbb{C}$ indeed gives the correct exit marginals from
$U_{n}$ and $U$, but the problem with this marginal identification
is that the quantity $\mathbf{E}\left(\big|Z_{\tau_{n}}-Z_{\tau}\big|^{p}\right)$
depends on the chosen coupling and need not coincide with the boundary
integral 
\[
\int^{\pi}_{-\pi}\big|G^{*}_{n}(e^{\theta i})-G^{*}(e^{\theta i})\big|^{p}\frac{d\theta}{2\pi}
\]
used in our proof. In fact, our proof crucially uses the disc exit-point
coupling (same $\xi\in\mathbb{S}^{1}$) which identifies the cost
with an $L^{p}$-distance of boundary traces.

However, this triggers the following question: Under what geometric
assumptions does $p$-convergence hold under the same-path planar
Brownian coupling? A typical first case is when 
\begin{equation}
U_{n}\underset{n\rightarrow+\infty}{\nearrow}U.\label{nested}
\end{equation}
We leave this question for future work.

\subsubsection*{Comment 3.}

The $p$-Brownian convergence introduced in this work applies to planar
domains and is formulated with exit times. However, in order to make
such a convergence genuine and meaningful, we required all the domains
(sequence $(U_{n})_{n}$ and target $U$) to share the same starting
point, which without loss of generality, we assumed to be the origin.
In fact, if the starting point, say $x$, lies outside both $U$ and
all $U_{n}$ then $\tau_{n}=\tau=0$ and $Z_{\tau}=Z^{[n]}_{\tau_{n}}=x$,
and hence the convergence conditions become trivial. We therefore
restrict to starting points
\[
x\in\Gamma:=U\cap\left(\bigcap_{n}U_{n}\right).
\]
In the same context, since the first exit time from a domain is the
first hitting time of its boundary, the notion of $p$-Brownian convergence
can be formulated for curves as well. We say that a sequence of planar
curves $(\gamma_{n})_{n}$ converges to a planar curve $\gamma$ if
there is a coupled pair $(Z^{[n]}_{\tau_{n}},\tau_{n})$ and $(Z_{\tau},\tau)$
(similar to the definition \ref{p-conv}) on a common probability
space such that 
\[
\begin{alignedat}{1}\mathbf{E}_{x}\left(\big|Z^{[n]}_{\tau_{n}}-Z_{\tau}\big|^{p}\right) & \underset{n\to+\infty}{\longrightarrow}0\end{alignedat}
,\,\,\,\begin{alignedat}{1}\mathbf{E}_{x}\left(\big|\tau_{n}-\tau\big|^{\frac{p}{2}}\right) & \underset{n\to+\infty}{\longrightarrow}0\end{alignedat}
\]
where
\[
\tau_{n}:=\inf\{t\mid Z^{[n]}_{t}\in\gamma_{n}\},\,\,\tau:=\inf\{t\mid Z_{t}\in\gamma\}
\]
provided the hitting times are almost surely finite. The starting
point $x$ is in the complement of $\gamma\cup\left(\bigcup_{n}\gamma_{n}\right)$.
We speculate that this extension for curves may be more challenging
to investigate compared to domains.\\
\\

Beyond these theoretical questions, there is also potential for applications.
The approximation of $\mu$-domains arises naturally in stochastic
analysis, numerical probability, and in fields such as quantitative
finance where Skorokhod-type embeddings are employed in model calibration.
Having a well-defined and numerically stable notion of domain convergence
may serve as a foundation for algorithms that require precise control
of exit distributions of Brownian motion. In this sense, the framework
presented here not only contributes to the theory of planar embeddings
but may also open the door to future applied developments. 

 \bibliographystyle{plain}
\bibliography{NumericalApproach}

@Book{revuz2013continuous,
  Title                    = {Continuous martingales and Brownian motion},
  Author                   = {D. Revuz and M. Yor},
  Publisher                = {Springer Science \& Business Media},
  Year                     = {2013},
  Volume                   = {293}
}

@InCollection{butzer1971hilbert,
  author    = {Butzer, P. and Nessel, R.},
  booktitle = {Fourier Analysis and Approximation},
  publisher = {Springer},
  title     = {Hilbert Transforms of Periodic Functions},
  year      = {1971},
  pages     = {334--354},
}

@Book{king2009hilbert,
  author    = {King, F.},
  publisher = {Cambridge University Press Cambridge},
  title     = {Hilbert transforms},
  year      = {2009},
}

@Article{burkholder1977exit,
  author    = {Burkholder, D.},
  title     = {Exit times of {B}rownian motion, harmonic majorization, and {H}ardy spaces},
  journal   = {Advances in {M}athematics},
  year      = {1977},
  volume    = {26},
  number    = {2},
  pages     = {182--205},
  publisher = {Academic Press},
}

@Book{duren2000theory,
  author    = {P. L Duren},
  publisher = {Courier Corporation},
  title     = {Theory of $H^p$ spaces},
  year      = {2000},
}

@Article{gross2019,
  author  = {Gross, R.},
  title   = {A conformal {S}korokhod embedding},
  journal = {Electronic Communications in Probability},
  year    = {2019},
}

@Book{Rudin2001,
  title     = {Real and complex analysis (3r ed.)},
  publisher = {McGraw-Hill Education},
  year      = {2001},
  author    = {Rudin, W.},
}

@Article{boudabra2019remarks,
  author  = {Boudabra, M. and Markowsky, G.},
  journal = {Electronic Communications in Probability},
  title   = {Remarks on {G}ross' technique for obtaining a conformal {S}korohod embedding of planar {B}rownian motion},
  year    = {2020},
}

@Article{Boudabra2020,
  author  = {M. Boudabra and G. Markowsky},
  journal = {Journal of Mathematical Analysis and Applications},
  title   = {A new solution to the conformal {S}korokhod embedding problem and applications to the {D}irichlet eigenvalue problem},
  year    = {2020},
  issn    = {0022-247X},
  number  = {2},
  pages   = {124351},
  volume  = {491},
}

@Article{Obloj2004,
  author    = {J. Ob\l\'{o}j},
  journal   = {Probability Surveys},
  title     = {The {S}korokhod embedding problem and its offspring},
  year      = {2004},
  pages     = {321 -- 392},
  publisher = {Institute of Mathematical Statistics and Bernoulli Society},
}

@Book{pandey2011hilbert,
  author    = {P. Jagdish Narayan},
  publisher = {John Wiley \& Sons},
  title     = {The Hilbert transform of Schwartz distributions and applications},
  year      = {2011},
  volume    = {27},
}

@PhdThesis{mcgovern1980hilbert,
  author = {J D. McGovern},
  title  = {The Hilbert Transform},
  year   = {1980},
}

@Book{billingsley2013convergence,
  author    = {P. Billingsley},
  publisher = {John Wiley \& Sons},
  title     = {Convergence of probability measures},
  year      = {2013},
}

@Book{resnick2008extreme,
  author    = {S. I Resnick},
  publisher = {Springer Science \& Business Media},
  title     = {Extreme values, regular variation, and point processes},
  year      = {2008},
  volume    = {4},
}

@InCollection{BurkholderDavisGundy1972,
  author    = {Burkholder, D. L. and Davis, B. J. and Gundy, R. F.},
  booktitle = {Proceedings of the Sixth Berkeley Symposium on Mathematical Statistics and Probability, Volume 2: Probability Theory},
  publisher = {University of California Press},
  title     = {Integral inequalities for convex functions of operators on martingales},
  year      = {1972},
  address   = {Berkeley, CA},
  pages     = {223--240},
}

@Article{Ruschendorf1985,
  author  = {L. R{\"u}schendorf},
  journal = {Zeitschrift f{\"u}r Wahrscheinlichkeitstheorie und Verwandte Gebiete},
  title   = {The {W}asserstein distance and approximation theorems},
  year    = {1985},
  number  = {1},
  pages   = {117--129},
  volume  = {70},
  doi     = {10.1007/BF00532240},
}

@Article{BOUDABRA2026,
  author  = {M. BOUDABRA},
  journal = {Bulletin of the Australian Mathematical Society},
  title   = {A Note on the planar Skorokhod embedding problem},
  year    = {2026},
  pages   = {1--7},
}

@Article{arias2002pointwise,
  author    = {J. ARIAS-DE-REYNA},
  journal   = {Journal of the London Mathematical Society},
  title     = {Pointwise convergence of Fourier series},
  year      = {2002},
  number    = {1},
  pages     = {139--153},
  volume    = {65},
  publisher = {Cambridge University Press},
}

@Article{fefferman1973pointwise,
  author    = {Fefferman, C.},
  journal   = {Annals of Mathematics},
  title     = {Pointwise convergence of {F}ourier series},
  year      = {1973},
  number    = {3},
  pages     = {551--571},
  volume    = {98},
  publisher = {JSTOR},
}

\end{document}